\documentclass[a4paper,11pt, abstract=true]{scrartcl}
\usepackage{german}
\usepackage[ansinew]{inputenc}
\usepackage[german,english,french,USenglish]{babel}
\usepackage{latexsym,amssymb,amsfonts,bbm}
\usepackage{theorem}
\usepackage{xcolor}
\usepackage{graphicx}
\usepackage{esint}
\usepackage{stmaryrd}
\newtheorem{dummytheorem}{Dummy-Theorem}[section]
\newcommand{\proofendsign}{$\Box$} 

\newtheorem{lemma}[dummytheorem]{Lemma}
\newtheorem{theorem}[dummytheorem]{Theorem}

\newenvironment{proof}{{\noindent \bf Proof }}
 {{\hspace*{\fill}\proofendsign\par\bigskip}}

\theorembodyfont{\normalfont}
\newtheorem{remarknorm}[dummytheorem]{Remark}
\newtheorem{examplenorm}[dummytheorem]{Example}

\newcommand{\V}{\mathbf{V}}
\newcommand{\N}{\mathbb{N}}
\newcommand{\Z}{\mathbb{Z}}

\newcommand{\R}{\mathbb{R}}

\newcommand{\F}{\mathbb{F}}
\newcommand{\D}{\mathbb{D}}

\newcommand{\I}{\mathbb{I}}

\newcommand{\oF}{\overline{F}}

\newcommand{\pr}{\mathbb{P}}
\newcommand{\ex}{\mathbb{E}}

\newcommand{\eins}{\mathbbm{1}}

\newcommand{\cadlag}{c\`adl\`ag}

\begin{document}


\title{Marcinkiewicz-Zygmund and ordinary strong laws for empirical distribution functions and plug-in estimators}
\author{Henryk Z\"ahle}
\date{\footnotesize Saarland University\\ Department of Mathematics\\ Postfach 151150\\ D-66041 Saarbr\"ucken\\ Germany\\
{\tt zaehle@math.uni-sb.de}}

\maketitle

\begin{abstract}
Both Marcinkiewicz-Zygmund strong laws of large numbers (MZ-SLLNs) and ordinary strong laws of large numbers (SLLNs) for plug-in estimators of general statistical functionals are derived. It is used that if a statistical functional is ``sufficiently regular'', then a (MZ-) SLLN for the estimator of the unknown distribution function yields a (MZ-) SLLN for the corresponding plug-in estimator. It is in particular shown that many L-, V- and risk functionals are ``sufficiently regular'', and that known results on the strong convergence of the empirical process of $\alpha$-mixing random variables can be improved. The presented approach does not only cover some known results but also provides some new strong laws for plug-in estimators of particular statistical functionals.
\end{abstract}



\newpage

\section{Introduction}\label{Introduction}

Let $\F$ be a class of distribution functions on the real line, and $T:\F\to\V'$ be a statistical functional, where $(\V',\|\cdot\|_{\V'})$ is a normed vector space. Let $(X_i)_{i\in\N}$ be a sequence of identically distributed real random variables on some probability space $(\Omega,{\cal F},\pr)$ with distribution function $F\in\F$. If $\widehat F_n$ denotes a reasonable estimator for $F$ based on the first $n$ observations $X_1,\ldots,X_n$, then $T(\widehat F_n)$ can provide a reasonable estimator for $T(F)$. In the context of nonparametric statistics, a central question concerns the rate of almost sure convergence of the plug-in estimator $T(\widehat F_n)$ to $T(F)$. That is, one wonders for which exponents $r'\ge 0$ the convergence
\begin{equation}\label{MZ-LLN-Introduction}
    n^{r'}\big\|T(\widehat F_n)-T(F)\|_{\V'}\,\longrightarrow\,0\qquad\pr\mbox{-a.s.}
\end{equation}
holds, where it is assumed that the left-hand side is ${\cal F}$-measurable for every $n\in\N$. This article is concerned with the convergence in (\ref{MZ-LLN-Introduction}) for both $r'>0$ and $r'=0$ and general statistical functionals $T$. In the case $r'>0$ the convergence in (\ref{MZ-LLN-Introduction}) can be seen as a Marcinkiewicz-Zygmund strong law of large numbers (MZ-SLLNs), and in the case $r'=0$ it can be seen as an ordinary strong law of large numbers (SLLNs).

Let $(\V,\|\cdot\|_\V)$ be a normed vector space with $\V$ a class of real functions on $\R$, and assume that the difference $F_1-F_2$ of every two distribution functions $F_1,F_2\in\F$  are elements of $\V$. So $\|\cdot\|_\V$ can in particular be seen as a metric on $\F$. Assume that $\widehat F_n$ is a $\F$-valued estimator for $F$ based on $X_1,\ldots,X_n$, that $\|\widehat F_n-F\|_\V$ is ${\cal F}$-measurable for every $n\in\N$, and that
\begin{equation}\label{MZ LLN for theta n}
    n^r\big\|\widehat F_n-F\big\|_\V\,\longrightarrow\,0\qquad\pr\mbox{-a.s.}
\end{equation}
for some $r\ge 0$. Finally, let $\widehat\F_n:=\{\widehat F_n(\omega):\omega\in\Omega\}$ be the range of $\widehat F_n$, and $\widehat\F$ be the union of the $\widehat\F_n$, $n\in\N$. Then, if for every sequence $(F_n)\subset\widehat\F$ with $\|F_n-F\|_\V\to 0$ we have that
\begin{eqnarray}\label{def eq for QC - def}
    \big\|T(F_n)-T(F)\big\|_{\V'}\,=\,{\cal O}\big(\|F_n-F\|_\V^\beta\big)
\end{eqnarray}
for some fixed $\beta>0$, we obtain by choosing $F_n:=\widehat F_n$ ($\omega$-wise) that (\ref{MZ-LLN-Introduction}) holds for $r'=r\beta$. If for every sequence $(F_n)\subset\widehat\F$ with $\|F_n-F\|_\V\to 0$ we only have that
\begin{eqnarray}\label{def eq for QC - def - ordinary}
    \big\|T(F_n)-T(F)\|_{\V'}\,=\,o(1),
\end{eqnarray}
then we obtain that (\ref{MZ-LLN-Introduction}) holds at least for $r'=0$; again choose $F_n:=\widehat F_n$ ($\omega$-wise). That is, in order to obtain a MZ-SLLN for $T(\widehat F_n)$ it suffices to have a MZ-SLLN for $\widehat F_n$ and to verify (\ref{def eq for QC - def}), and in order to obtain a SLLN for $T(\widehat F_n)$ it suffices to have a SLLN for $\widehat F_n$ and to verify (\ref{def eq for QC - def - ordinary}). We refer to (\ref{def eq for QC - def}) as {\em H\"older-$\beta$ continuity of $T$ at $F$ w.r.t.\ $(\|\cdot\|_\V,\|\cdot\|_{\V'})$ and $\widehat\F$}, and to (\ref{def eq for QC - def - ordinary}) as {\em continuity of $T$ at $F$ w.r.t.\ $(\|\cdot\|_\V,\|\cdot\|_{\V'})$ and $\widehat\F$}.

Concerning $\widehat F_n$ we will restrict ourselves to the empirical distribution function. That is, from now on we assume that $\widehat F_n=\frac{1}{n}\sum_{i=1}^n\eins_{[X_i,\infty)}$. In particular, $\widehat\F$ will always be contained in the class of all empirical distribution functions $\frac{1}{n}\sum_{i=1}^n\eins_{[x_i,\infty)}$ with $n\in\N$ and $x_1,\ldots,x_n\in\R$. The rest of the article is organized as follows. In Section \ref{Strong laws for widehat F n} we will first present some results that illustrate (\ref{MZ LLN for theta n}) for uniform and nonuniform sup-norms. Thereafter, in Section \ref{Strong laws for widehat T F n}, we will show that several statistical functionals are (H\"older-$\beta$) continuous w.r.t.\ uniform or nonuniform sup-norms. The proofs of the results of Section \ref{Strong laws for widehat F n} will be given in Sections \ref{appendix a1}--\ref{appendix a2}.


\section{Strong laws for $\widehat F_n$}\label{Strong laws for widehat F n}

An intrinsic example for $(\V,\|\cdot\|_\V)$ is the normed vector space $(\D_\phi,\|\cdot\|_\phi)$ of all \cadlag\ functions $\psi$ with $\|\psi\|_\phi<\infty$, where $\|\psi\|_\phi:=\|\psi\phi\|_\infty$ refers to the nonuniform sup-norm based on some weight function $\phi$. By {\it weight function} we mean any continuous function $\phi:\R\to\R_+$ which is bounded away from zero, i.e.\ $\phi(\cdot)\ge\varepsilon$ for some $\varepsilon>0$, and u-shaped, i.e.\ nonincreasing on $(-\infty,x_\phi]$ and nondecreasing on $[x_\phi,\infty)$ for some $x_\phi\in\R$. In Section \ref{Strong laws for widehat T F n} we will see that many statistical functionals are (H\"older-$\beta$) continuous w.r.t.\ $(\|\cdot\|_\phi,|\cdot|)$  and $\widehat\F$. Here we will first present some results that illustrate (\ref{MZ LLN for theta n}) for $\|\cdot\|_\V=\|\cdot\|_\phi$.

We begin with the case of independent observations. The following result strongly relies on \cite[Theorem 7.3]{Andersen et al 1988}. The proof can be found in Section \ref{appendix a1}.

\begin{theorem}\label{generalized GC}
Let $(X_i)$ be an i.i.d.\ sequence of random variables with distribution function $F$. Let $\phi$ be a weight function and $r\in[0,\frac{1}{2})$. If $\int_{-\infty}^\infty\phi(x)^{1/(1-r)}dF(x)<\infty$,
then
$$
    n^{r}\|\widehat F_n-F\|_{\phi}\longrightarrow 0\qquad\pr\mbox{-a.s.}
$$
\end{theorem}

Let us now turn to the case of weakly dependent data. We will assume that the sequence $(X_i)$ is $\alpha$-mixing in the sense of \cite{Rosenblatt1956}, i.e.\ that the mixing coefficient $\alpha(n):=\sup_{k\ge 1}\sup_{A,B}\,|\pr[A\cap B]-\pr[A]\pr[B]|$ converges to zero as $n\to\infty$, where the second supremum ranges over all $A\in\sigma(X_1,\ldots,X_k)$ and $B\in\sigma(X_{k+n},X_{k+n+1},\ldots)$. For an overview on mixing conditions see \cite{Bradley2005,Doukhan1994}.

\begin{theorem}\label{generalized GC - dependent}
Let $(X_i)$ be a sequence of identically distributed random variables with distribution function $F$. Suppose that $(X_i)$ is $\alpha$-mixing with mixing coefficients $(\alpha(n))$. Let $r\in[0,\frac{1}{2})$ and assume that $\alpha(n)\le K n^{-\vartheta}$ for all $n\in\N$ and some constants $K>0$ and $\vartheta>2r$. Then
\begin{equation}\label{generalized GC - dependent - eq}
    n^{r}\|\widehat F_n-F\|_\infty\longrightarrow 0\qquad\pr\mbox{-a.s.}
\end{equation}
\end{theorem}

{For the proof of Theorem \ref{generalized GC - dependent}, which can be found in Section \ref{appendix a0}, we will combine arguments of \cite{Philipp1977} and \cite{Rio2000}. Under the stronger mixing conditions $\alpha(n)\le K n^{-8}$ and $\alpha(n)\le K n^{-(3+\varepsilon)}$, $\varepsilon>0$, the convergence in (\ref{generalized GC - dependent - eq}) is already known from \cite{BerkesPhilipp1977,Philipp1977} and \cite{Yoshihara1979}, respectively. If in (\ref{generalized GC - dependent - eq}) almost sure converges is replaced by convergence in probability, then the result is known from \cite{Yu1994}. The more recent article \cite{BerkesHoermannSchauer2009} contains a version of Theorem \ref{generalized GC - dependent} for empirical processes of so called $S$-mixing sequences. The concept of $S$-mixing seems to be less restrictive than the concept of $\alpha$-mixing, but the two concepts are not directly comparable.

To compare Theorem \ref{generalized GC - dependent} above with Theorem 1 in \cite{BerkesHoermannSchauer2009} anyway, let $X_t:=\sum_{s=0}^\infty a_{s}Z_{t-s}$, $t\in\N$, be a linear process with $(Z_s)_{s\in\Z}$ a sequence of i.i.d.\ random variables with expectation zero, a finite absolute $p$th moment for some $p\ge 2$, and a Lebesgue density $f$ satisfying $\int|f(x)-f(y)|\,dx\le M |x-y|$ for all $x,y\in\R$ and some finite constant $M>0$. For instance, these conditions are fulfilled when $Z_0$ is centered normal. 
If $a_s=s^{-\gamma}$ for some $\gamma>(2+p)/p$, then results in \cite{Gorodetskii1977} show that $(X_t)$ is $\alpha$-mixing with $\alpha(n)\le K\,n^{-\vartheta}$ for $\vartheta=(p(\gamma-1)-2)/(1+p)$. So, if we choose $\gamma=(3+2p)/p$, then we have $\vartheta=1$ and therefore Theorem \ref{generalized GC - dependent} yields
\begin{equation}\label{generalized GC - dependent - eq - two}
    n^{r}\|\widehat F_n-F\|_\infty\longrightarrow 0\qquad\pr\mbox{-a.s.},\qquad\forall\,r\in[0,1/2).
\end{equation}
On the other hand, in order to obtain (\ref{generalized GC - dependent - eq - two}) with the help of Theorem 1 and the considerations in Section 3.1 of \cite{BerkesHoermannSchauer2009}, one has to choose $\gamma=(A+(A+1)p)/p$ for some $A>4$. Since $(A+(A+1)p)/p>(3+2p)/p$ for every $A>4$, Theorem \ref{generalized GC - dependent} above appears to be less restrictive in the $\alpha$-mixing case than Theorem 1 in \cite{BerkesHoermannSchauer2009}. On the other hand, Theorem 1 in \cite{BerkesHoermannSchauer2009} covers even the two-parameter empirical process.

It seems to be hard to modify the arguments of the proof of Theorem \ref{generalized GC - dependent} in such a way that they can be applied to the case of a nonconstant weight function $\phi$. To the best of the author's knowledge, there is no respective results in the literature so far. Results of \cite{DedeckerMerlevede2007} cover the case where in (\ref{generalized GC - dependent - eq}) the sup-norm is replaced by the $L^p$-norm w.r.t.\ a $\sigma$-finite measure for $p>1$. However, as the case $p=1$ is excluded, the results do not cover the $L^1$-Wasserstein distance. Notice that several statistical functionals can be shown to be continuous w.r.t.\ the $L^1$-Wasserstein distance.

If one is content with $r=0$, i.e.\ with an ordinary SLLN, then the following Theorem \ref{generalized GC - dependent - ordinary} gives a respective result for nonconstant weight functions $\phi$ and $\alpha$-mixing data. The proof of Theorem \ref{generalized GC - dependent - ordinary} can be found in Section \ref{appendix a2}. To the best of the author's knowledge, Theorem \ref{generalized GC - dependent - ordinary} provides the first result on the strong convergence of the empirical distribution function $\widehat F_n$ of $\alpha$- mixing random variables to the underlying distribution function $F$ w.r.t.\ a {\em nonuniform} sup-norm. For any nonincreasing function $h:\R_+\to[0,1]$, we let $h^\rightarrow(y):=\sup\{x\in\R_+:h(x)>y\}$, $y\in[0,1]$, be its right-continuous inverse, with the convention $\sup\emptyset:=0$.

\begin{theorem}\label{generalized GC - dependent - ordinary}
Let $(X_i)$ be a sequence of identically distributed random variables with distribution function $F$. Let $\phi$ be a weight function, and suppose that  $\int_{-\infty}^\infty\phi\,dF<\infty$. Suppose that $(X_i)$ is $\alpha$-mixing with mixing coefficients $(\alpha(n))$, let $\alpha(t):=\alpha(\lfloor t\rfloor)$ be the \cadlag\ extension of $\alpha(\cdot)$ from $\N$ to $\R_+$, and assume that
\begin{equation}\label{generalized GC - dependent - ordinary - cond}
    \int_0^1 \log\big(1+\alpha^\rightarrow(s/2)\big)\,\overline{G}\,^\rightarrow(s)\,ds\,<\,\infty
\end{equation}
for $\overline G:=1-G$, where $G$ denotes the distribution function of $\phi(X_1)$. Then
\begin{equation}\label{generalized GC - dependent - ordinary - eq}
    \|\widehat F_n-F\|_\phi\longrightarrow 0\qquad\pr\mbox{-a.s.}
\end{equation}
\end{theorem}

\begin{remarknorm}\label{generalized GC - dependent - ordinary - remark}
Notice that (\ref{generalized GC - dependent - ordinary - cond}) holds in particular if $\ex[\phi(X_1)\log^+\phi(X_1)]<\infty$ and $\alpha(n)={\cal O}(n^{-\vartheta})$ for some arbitrarily small $\vartheta>0$; cf.\ \cite[Application 5, p.\,924]{Rio1995}.
{\hspace*{\fill}$\Diamond$\par\bigskip}
\end{remarknorm}


\section{Strong laws for $T(\widehat F_n)$ for particular functionals $T$}\label{Strong laws for widehat T F n}

In this section we will show that several statistical functionals $T$ are continuous w.r.t.\ nonuniform sup-norms $\|\cdot\|_\phi$ or w.r.t.\ the uniform sup-norm $\|\cdot\|_\infty$. As a consequence we will obtain MZ-SLLNs and SLLNs for $T(\widehat F_n)$, cf.\ the discussion in the Introduction.


\subsection{L-functionals}\label{l statistics}

Let $K$ be the distribution function of a probability measure on $([0,1],{\cal B}([0,1]))$, and $\F_K$ be the class of all distribution function $F$ on the real line for which $\int_{-\infty}^\infty|x|\,dK(F(x))<\infty$. The functional ${\cal L}$, defined by
\begin{equation}\label{def L functional}
    {\cal L}(F):={\cal L}_K(F):=\int_{-\infty}^\infty x\,dK(F(x)),\qquad F\in \F_K,
\end{equation}
is called L-functional associated with $K$. It was shown in \cite{BeutnerZaehle2010} that if $K$ is continuous and piecewise differentiable, the (piecewise) derivative $K'$ is bounded above and $F\in\F_K$ takes the value $d\in(0,1)$ at most once if $K$ is not differentiable at $d$, then for every $\lambda>1$ the functional ${\cal L}:\F_K\to\R$ is quasi-Hadamard differentiable at $F$ tangentially to $\D_{\phi_\lambda}$, where $\phi_\lambda(x):=(1+|x|)^\lambda$. This implies in particular that ${\cal L}$ is also H\"older-$1$ continuous at $F$ w.r.t.\ $(\|\cdot\|_{\phi_\lambda},|\cdot|)$ and $\widehat\F$. The assumption that $K'$ be bounded can be relaxed at the cost of a more sophisticated choice of the weight function $\phi$\,; cf.\ the following Lemma \ref{HD of L functional}. In the lemma we will assume without loss of generality that $F(x)\in(0,1)$ for all $x\in\R$. If $F$ reaches $0$ or $1$, then the weight function $\phi_{\gamma,F}$, defined in (\ref{HD of L functional - def weight fct}) below, can be modified in the obvious way.

\begin{lemma}\label{HD of L functional}
Let $F\in\F_K$, $\oF:=1-F$, $0\le \beta'<\gamma\le 1$, and assume that
\begin{itemize}
    \item[(a)] $K$ is locally Lipschitz continuous at $x$ with local Lipschitz constant $L(x)>0$ for all $x\in(0,1)$,and $L(x)\le C'x^{-\beta'}(1-x)^{-\beta'}$ for all $x\in(0,1)$ and some constant $C'>0$.
    \item[(b)] $\int_{-\infty}^0 F(x)^{\gamma-\beta'}dx+\int_0^{\infty}\oF(x)^{\gamma-\beta'}dx<\infty$.
\end{itemize}
Assume $F(x)\in(0,1)$ for all $x\in\R$, and define the weight function
\begin{equation}\label{HD of L functional - def weight fct}
    \phi_{\gamma,F}(x):=F(x)^{-\gamma}\eins_{(-\infty,0)}(x)+{\oF}(x)^{-\gamma}\eins_{[0,\infty)}(x),\qquad x\in\R.
\end{equation}
Then the functional ${\cal L}:\F_K\to\R$ is H\"older-$1$ continuous at $F$ w.r.t.\ $(\|\cdot\|_{\phi_{\gamma,F}},|\cdot|)$ and $\widehat\F$.
\end{lemma}

\begin{proof}
Since ${\cal L}(F)$ can be written as ${\cal L}(F)=-\int_{-\infty}^0 K(F(x))\,dx+\int_0^\infty(1-K(F(x)))\,dx$, we obtain by assumption (a)
\begin{eqnarray*}\label{Hadamard differentiability - proof - 1}
    |{\cal L}(F_n)-{\cal L}(F)|
    & \le & \int_{-\infty}^\infty\big|K(F_n(x))-K(F(x))\big|\,dx\\
    & \le & \int_{-\infty}^\infty L(F(x))\,\big|(F_n-F)(x)\big|\,dx\\
    & \le & \Big(C'\int_{-\infty}^\infty F(x)^{-\beta'}\oF(x)^{-\beta'}\,\phi_{\gamma,F}(x)^{-1}\,dx\Big)\|F_n-F\|_{\phi_{\gamma,F}}
\end{eqnarray*}
for every sequence $(F_n)\subset\widehat\F$; notice that $\|F_n-F\|_{\phi_{\gamma,F}}$ is finite because of $\gamma\le 1$. Since the latter integral is finite by assumption (b), we obtain $|{\cal L}(F_n)-{\cal L}(F)|={\cal O}(\|F_n-F\|_{\phi_{\gamma,F}})$ when $\|F_n-F\|_{\phi_{\gamma,F}}\to 0$.
\end{proof}

\begin{remarknorm}\label{example for df K}
Assumption (a) in Lemma \ref{HD of L functional} is fulfilled for every continuous {\it convex} distribution function $K$ on the unit interval satisfying $1-K(x)\le C(1-x)^\beta$ (for all $x\in[0,1]$ and some $C>0$) with $\beta=1-\beta'$ and $0\le \beta'<1$. In this case we can choose $L(x)=C(1-x)^{-\beta'}$ and $C'=C$. For instance, $K(x):=1-(1-x)^\beta$ provides such a distribution function when $0<\beta\le 1$.
{\hspace*{\fill}$\Diamond$\par\bigskip}
\end{remarknorm}

\begin{remarknorm}
Lemma \ref{HD of L functional} shows that the functional ${\cal L}$ is H\"older-$1$ continuous at $F$ when $K$ is locally Lipschitz continuous on $(0,1)$ and at least H\"older continuous (of a certain order) at $0$ and $1$. If the kernel $K$ is only {\em piecewise} H\"older-$\beta$ continuous on $[0,1]$ for some $\beta\in(0,1)$, and $F\in\F_K$ satisfies $\|F-\eins_{[0,\infty)}\|_{\phi_\gamma}<\infty$ for some $\gamma>1/\beta$, then one can derive at least H\"older-$\beta$ continuity of ${\cal L}$ at $F$ w.r.t.\ $(\|\cdot\|_{\phi_{\lambda}},|\cdot|)$ and $\widehat\F$; cf.\ \cite[Theorem 2]{Zaehle2010}.
{\hspace*{\fill}$\Diamond$\par\bigskip}
\end{remarknorm}

\begin{theorem}\label{SLLN L abh - eq}
Let $X_1,X_2,\ldots$ be identically distributed random variables with distribution function $F\in\F_K$. Let $0\le \beta'<\gamma\le 1$, and assume that conditions (a)--(b) of Lemma \ref{HD of L functional} hold.
\begin{itemize}
    \item[(i)] If the $X_i$ are independent and $F$ satisfies the assumptions of Theorem \ref{generalized GC} for $r\in[0,\frac{1}{2})$ and $\phi=\phi_{\gamma,F}$ defined in (\ref{HD of L functional - def weight fct}), then we have $n^r|{\cal L}(\widehat F_n)-{\cal L}(F)|\to 0$ $\pr$-a.s.
    \item[(ii)] If the sequence $(X_i)$ is $\alpha$-mixing and satisfies the assumptions of Theorem \ref{generalized GC - dependent - ordinary} for $\phi=\phi_{\gamma,F}$ defined in (\ref{HD of L functional - def weight fct}), then we have at least $|{\cal L}(\widehat F_n)-{\cal L}(F)|\to 0$ $\pr$-a.s.
\end{itemize}
\end{theorem}

In view of Lemma \ref{HD of L functional} and the discussion in the Introduction, assertions (i) and (ii) in Theorem \ref{SLLN L abh - eq} are immediate consequences of Theorems \ref{generalized GC} and \ref{generalized GC - dependent - ordinary}, respectively. Example \ref{example for l functional} below sheds light on the assumptions of Theorem \ref{SLLN L abh - eq}. Part (i) of Theorem \ref{SLLN L abh - eq} recovers results from \cite{Bose1998,Helmers1977a,Wellner1977b,Zaehle2010}. Ordinary SLLNs for L-statistics in the fashion of part (ii) of Theorem \ref{SLLN L abh - eq} can be found in \cite{van Zwet 1980} for i.i.d.\ data, in \cite{Baklanov2006} for $\phi$-mixing data, and in \cite{Aaronson et al 1996,Baklanov2006,GilatHelmers1997} for ergodic stationary data. In the case of $\alpha$-mixing data the conditions in \cite{Baklanov2006,GilatHelmers1997} are comparable to those of part (ii) in Theorem \ref{SLLN L abh - eq}. That is, the statements of Theorem \ref{SLLN L abh - eq} are basically already known. Nevertheless our approach leads to simple proofs  once Theorems \ref{generalized GC} and \ref{generalized GC - dependent - ordinary} are established. In the context of general law-invariant risk measures, in Section \ref{risk measures} below, we will also take advantage of the method of proof of Theorem \ref{SLLN L abh - eq}.

\begin{examplenorm}\label{example for l functional}
Let $0\le \beta'<\gamma\le 1$, and assume that condition (a) in Lemma \ref{HD of L functional} holds. Further assume that $F(x)=c_1|x|^{-\alpha}$ for all $x\le -x_0$, and $\oF(x)=c_2 x^{-\alpha}$ for all $x\ge x_0$, for some constants $\alpha,x_0,c_1,c_2>0$. In this case, assumption (b) in Lemma \ref{HD of L functional} and the integrability condition in Theorem \ref{generalized GC} (with $\phi=\phi_{\gamma,F}$) read as
\begin{equation}\label{example for l functional - eq 1}
    \int_{-\infty}^{-1}|x|^{-\alpha(\gamma-\beta')}\,dx+\int_1^\infty x^{-\alpha(\gamma-\beta')}\,dx\,<\,\infty
\end{equation}
and
\begin{equation}\label{example for l functional - eq 2}
    \int_{-\infty}^{-1}|x|^{\frac{\alpha\gamma}{1-r}-\alpha-1}\,dx+\int_1^\infty x^{\frac{\alpha\gamma}{1-r}-\alpha-1}\,dx\,<\,\infty,
\end{equation}
respectively. Condition (\ref{example for l functional - eq 1}) holds if and only if $\gamma>\beta'+\frac{1}{\alpha}$, and condition (\ref{example for l functional - eq 2}) holds if and only if $\gamma<1-r$. So, if we assume $0\le\beta'+\frac{1}{\alpha}<1-r$ and $0\le r<\frac{1}{2}$, then the assumptions on $K$ and $F$ imposed in the setting of part (i) of Theorem \ref{SLLN L abh - eq} are fulfilled (with any $\gamma\in(\beta'+\frac{1}{\alpha},1-r)$). In particular, if we assume $0\le\beta'+\frac{1}{\alpha}<1$, then the assumptions on $K$ and $F$ imposed in the setting of part (ii) of Theorem \ref{SLLN L abh - eq} are fulfilled (with any $\gamma\in(\beta'+\frac{1}{\alpha},1)$).
{\hspace*{\fill}$\Diamond$\par\bigskip}
\end{examplenorm}

In the following theorem we restrict ourselves to empirical quantile estimators based on $\alpha$-mixing data. However, it can easily be extended to plug-in estimators of more general L-functionals ${\cal L}_K$ with $dK$ having compact support strictly within $(0,1)$. Under the stronger mixing conditions $\alpha(n)\le K e^{-\varepsilon n}$, $\varepsilon>0$, and $\alpha(n)\le K n^{-8}$ the result of Theorem \ref{lambda - regularity of rho} is basically already known from \cite{BabuSingh1978} and \cite{Wendler2012}, respectively. We let $H^{\leftarrow}(x):=\inf\{y\in\R\,:\,H(y)\ge x\}$, $x\in\R$, denote the left-continuous inverse of any nondecreasing function $H:\R\to\R$, with the convention $\inf\emptyset:=\infty$.

\begin{theorem}\label{lambda - regularity of rho}
Let $(X_i)$ be an $\alpha$-mixing sequence of identically distributed random variables with distribution function $F$. Let $r\in[0,\frac{1}{2})$, and assume that the mixing coefficients satisfy $\alpha(n)\le K n^{-\vartheta}$ for all $n\in\N$ and some constants $K>0$, $\vartheta>2r$. Let $y\in(0,1)$, and assume that $F$ is differentiable at $F^\leftarrow(y)$ with $F'(F^\leftarrow(y))>0$.  Then, $n^r|\widehat F_n^\leftarrow(y)-F^\leftarrow(y)|\to 0$ $\pr$-a.s.
\end{theorem}

\begin{proof}
Since $F^{\leftarrow}(y)={\cal L}_{K_y}(F)$ with $K_y=\eins_{[y,1]}$, the proof of Theorem 2 in \cite{Zaehle2010} shows that, under the above assumptions on $F$, $\pr$-a.s.\ there is some constant $C>0$ such that $|\widehat F_n^{\leftarrow}(y)-F^{\leftarrow}(y)|\le C\,\|\widehat F_n-F\|_\infty$ for all $n\in\N$. Now the claim follows directly from Theorem \ref{generalized GC - dependent}.
\end{proof}


\subsection{Law-invariant coherent risk measures}\label{risk measures}

Let $\rho$ be a law-invariant coherent risk measure on ${\cal X}:={\cal L}^p(\Omega,{\cal F},\pr)$ for some $p\in[1,\infty]$, i.e.\ $\rho$ be a mapping from ${\cal X}$ to $\R$ being
\begin{itemize}
    \item monotone: $\rho(X)\le\rho(Y)$ for all $X,Y\in{\cal X}$ with $X\le Y$ $\pr$-almost surely,
    \item translation-equivariant: $\rho(X+m)=\rho(X)+m$ for all $X\in{\cal X}$ and $m\in\R$,
    \item subadditive: $\rho(X+Y)\le\rho(X)+\rho(Y)$ for all $X,Y\in{\cal X}$,
    \item positively homogenous: $\rho(\lambda X)=\lambda\rho(X)$ for all $X\in{\cal X}$ and $\lambda\ge 0$.
\end{itemize}
Since $\rho$ is law-invariant, we may regard it as a functional ${\cal R}$ on the set $\F^p$ of all distribution functions of random variables in ${\cal L}^p(\Omega,{\cal F},\pr)$. If the underlying probability space  $(\Omega,{\cal F},\pr)$ is rich enough to support a random variable with continuous distribution (which is equivalent to $(\Omega,{\cal F},\pr)$ being atomless in the sense of \cite[Definition A.26]{FoellmerSchied2004}), then the functional ${\cal R}$ admits the representation
\begin{equation}\label{robust representation}
    {\cal R}(F) = \sup_{K\in{\cal K}_{\cal R}}\,{\cal L}_K(F)\qquad\forall\,F\in\F^p,
\end{equation}
where ${\cal L}_K$ is the L-functional associated with kernel $K$ (cf.\ (\ref{def L functional})) and ${\cal K}_{\cal R}$ is a suitable set of continuous convex distribution functions on the unit interval. This was shown in \cite[Corollary 4.72]{FoellmerSchied2004} for $p=\infty$, and in \cite{KraetschmerZaehle2011} for the general case. Notice that in \cite{KraetschmerZaehle2011} the role of $K$ is played by $\check g$.

If condition (a) in Lemma \ref{HD of L functional} holds for every $K\in{\cal K}_{\cal R}$ with the same $L(x),\beta',C'$, and $F\in\F^p$ satisfies condition (b) in Lemma \ref{HD of L functional}, then, in view of
$$
    |{\cal R}(F_n) - {\cal R}(F)|
    \,= \, \Big|\sup_{K\in{\cal K}_{\cal R}}{\cal L}_K(F_n)-\sup_{K\in{\cal K}_{\cal R}}{\cal L}_K(F)\Big|
    \, \le \ \sup_{K\in{\cal K}_{\cal R}}|{\cal L}_K(F_n) - {\cal L}_K(F)|
$$
the proof of Lemma \ref{HD of L functional} shows that the functional ${\cal R}:\F^p\to\R$ is H\"older-$1$ continuous at $F$ w.r.t.\ $(\|\cdot\|_{\phi_{\gamma,F}},|\cdot|)$ and $\widehat\F$. So, in this case assertions (i)--(ii) in Theorem \ref{SLLN L abh - eq} also hold for ${\cal R}$ in place of ${\cal L}$. This seems to be the first general respective result in the context of law-invariant coherent risk measures.

\begin{examplenorm}\label{example one sided moments}
It is easy to show that
$$
    \rho_{p,a}(X)\,:=\,\ex[X]+a\,\ex[((X-\ex[X])^{+})^p]^{1/p} 
$$
provides a law-invariant coherent risk measure (called risk measure based on one-sided moments) on ${\cal L}^{p}(\Omega,{\cal F},\pr)$ for every $p\in[1,\infty)$ and $a\in[0,1]$. It was shown in \cite[Lemma A.5]{KraetschmerZaehle2011} that the associated functional ${\cal R}_{p,a}:\F^p\to\R$ is not a L-functional when $a>0$. But according to our preceding discussion ${\cal R}_{p,a}$ can be represented as in (\ref{robust representation}). We clearly have
\begin{eqnarray*}
    1-K(x)
    & = & {\cal L}_K(F_{B_{1-x}})\\
    & \le & {\cal R}_{p,a}(F_{B_{1-x}})\\
    & = & (1-x)+a((1-x)x^p)^{1/p}\\
    & \le & (1+a)(1-x)^{1/p}
\end{eqnarray*}
(where $F_{B_{1-x}}$ is the Bernoulli distribution function with expectation $1-x$) for all $x\in(0,1)$ and $K\in{\cal K}_{{\cal R}_{p,a}}$. Thus Remark \ref{example for df K} and the preceding discussion show that the risk functional ${\cal R}_{p,a}$ is H\"older-$1$ continuous at $F\in\F^p$ w.r.t.\ $(\|\cdot\|_{\phi_{\gamma,F}},|\cdot|)$ and $\widehat\F$, provided $F$ satisfies condition (b) in Lemma \ref{HD of L functional} with $\beta'=1-\frac{1}{p}$.
{\hspace*{\fill}$\Diamond$\par\bigskip}
\end{examplenorm}


\subsection{V-functionals}\label{u and v statistics}

Let $g:\R^2\to\R$ be a measurable function, and $\F_g$ be the class of all distribution functions $F$ on the real line for which $\int_{-\infty}^\infty\int_{-\infty}^\infty|g(x_1,x_2)|dF(x_1)dF(x_2)<\infty$. The functional ${\cal V}$, defined by
$$
    {\cal V}(F):={\cal V}_g(F):=\int_{-\infty}^\infty\int_{-\infty}^\infty g(x_1,x_2)\,dF(x_1)dF(x_2),\qquad F\in\F_g,
$$
is called von Mises-functional (or simply V-functional) associated with $g$. It was shown in \cite{BeutnerZaehle2012} that under fairly weak assumptions on $g$ and $F\in\F_g$ the functional ${\cal V}$ is H\"older-$1$ continuous at $F$ w.r.t.\ $(\|\cdot\|_\phi,|\cdot|)$ and $\widehat\F$. Thus, from Theorems \ref{generalized GC}--\ref{generalized GC - dependent - ordinary} one can easily derive MZ-SLLNs and SLLNs for ${\cal V}(\widehat F_n)$; see also \cite{BeutnerZaehle2012}.

MZ-SLLNs for i.i.d.\ data that can be obtained with the help of Theorem \ref{generalized GC} are already known from  \cite{GineZinn1992,Sen1974,Teicher1998}. Related ordinary SLLNs can be found in \cite{Hoeffding1961} for i.i.d.\ data, in \cite{Wang1995} for $\phi^*$-mixing data, in \cite{Arcones1998} for $\beta$-mixing data, and in \cite{Aaronson et al 1996} for ergodic stationary data. The proofs in \cite{Wang1995} contain gaps as was revealed in \cite[p.\,14]{Arcones1998}. The conditions on $g$, $F$ and the mixing coefficients in \cite[Theorem 1]{Arcones1998} are comparable to those under which Theorem \ref{generalized GC - dependent - ordinary} and Remark \ref{generalized GC - dependent - ordinary - remark} above yield ordinary SLLNs, but in our setting we can consider even $\alpha$-mixing. The assumptions on the kernel $g$ in \cite{Aaronson et al 1996} are more restrictive than the conditions we would have to impose in our setting. On the other hand, ergodicity is a weaker assumption than $\alpha$-mixing.

To the best of the author's knowledge, so far MZ-SLLNs for weakly dependent data can be found only in \cite{DehlingSharipov2009}. In \cite{DehlingSharipov2009} the data are assumed to be $\beta$-mixing. In the case of a bounded kernel $g$, Theorem 1 in \cite{DehlingSharipov2009} assumes that the mixing coefficients satisfy $\sum_{n=1}^\infty n\beta(n)<\infty$ in order to obtain a MZ-SLLN for any $r\in[0,1/2)$. With the help of Theorem \ref{generalized GC - dependent} above this condition can be relaxed to $\alpha(n)={\cal O}(n^{-1})$, even in the less restrictive case of $\alpha$-mixing. On the other hand, Theorem 2 in \cite{DehlingSharipov2009} covers also the case of unbounded kernels $g$.

It was shown in \cite{BeutnerZaehle2012} that V-functionals that are degenerate w.r.t.\ $(g,F)$ are typically even H\"older-$2$ continuous at $F$ w.r.t.\ $(\|\cdot\|_\phi,|\cdot|)$ and $\widehat\F$. So the rate of convergence of degenerate V-statistics is typically twice the rate of convergence of non-degenerate V-statistics; for details see again \cite{BeutnerZaehle2012}.


\section{Proof of Theorem \ref{generalized GC}}\label{appendix a1}

By the usual quantile transformation \cite[p.103]{ShorackWellner1986}, we may and do choose a sequence of i.i.d.\ $U[0,1]$-random variables, possibly on an extension $(\overline\Omega,\overline{\cal F},\overline\pr)$ of the original probability space $(\Omega,{\cal F},\pr)$, such that the corresponding empirical distribution function $\widehat G_n$ satisfies $\widehat F_n=\widehat G_n(F)$ $\overline\pr$-a.s.
Then
\begin{eqnarray*}
    n^{r}\|\widehat F_n-F\|_\phi
    & = & n^{r}\,\sup_{x\in\R}|\widehat G_n(F(x))-F(x)|\,\phi(x)\\
    & \le & n^{r}\sup_{s\in(0,1)}|\widehat G_n(s)-s|\,w(s)
\end{eqnarray*}
with $w(s):=\phi(\max\{F^\leftarrow(s);F^\rightarrow(s)\})$, where $F^\leftarrow$ and $F^\rightarrow$ denote the left- and the right-continuous inverse of $F$, respectively. According to Theorem 7.3\,(3) in \cite{Andersen et al 1988}, the latter bound converges $\overline\pr$-a.s.\ to $0$ as $n\to\infty$ if and only if $\int_{(0,1)}w(s)^{1/(1-r)}ds<\infty$. Since $\int_{(0,1)}w(s)^{1/(1-r)}\,ds=\int_{\R}\phi(x)^{1/(1-r)}\,dF(x)$ by a change-of-variable (and the fact that $F^\leftarrow=F^\rightarrow$ $ds$-almost everywhere), and since this integral is finite by assumption, we thus obtain $n^{r}\|F_n-F\|_\phi\to 0$ $\overline\pr$-a.s.


\section{Proof of Theorem \ref{generalized GC - dependent}}\label{appendix a0}

In this section, we will prove Theorem \ref{generalized GC - dependent}. By the usual quantile transformation \cite[p.103]{ShorackWellner1986} (which works also for mixing data) it suffices to prove the result in the special case of $U[0,1]$-distributed random variables. 
Let $(U_i)=(U_i)_{i\in\N}$ be an $\alpha$-mixing sequence of identically $U{[0,1]}$-distributed random variables on some probability space $(\Omega,{\cal F},\pr)$. Let $\I$ be the identity on $[0,1]$, and $\widehat G_n:=\frac{1}{n}\sum_{i=1}^n\eins_{[U_i,1]}$ be the empirical distribution function of $U_1,\dots,U_n$.

\begin{theorem}\label{mzlln for ep unif} 
Let $r\in[0,1/2)$, $C>0$ and $\vartheta>2r$. Suppose that the mixing coefficients $(\alpha(n))$ of the sequence $(U_i)$ satisfy $\alpha(n)\le Cn^{-\vartheta}$ for all $n\in\N$. Then
\begin{equation}\label{mzlln for ep unif - eq 2}
    n^{r}\|\widehat G_n-\I\|_\infty\longrightarrow 0\qquad\pr\mbox{-a.s.}
\end{equation}
\end{theorem}

The proof of Theorem \ref{mzlln for ep unif} will be carried out in three steps (Sections \ref{Auxiliary results, Part I}--\ref{proof of mzlln for ep -- actual proof}). For every $p\in\N_0$, $q\in\N$ and $t\in[0,1]$, define
$$
    Z_{p,q}(t)\,:=\,\Big|\sum_{i=p+1}^{p+q}\Big(\eins_{[0,t]}(U_i)-t)\Big)\Big|.
$$
Thus, in order to verify (\ref{mzlln for ep unif - eq 2}), we have to show
\begin{equation}\label{convergence of interest - modi}
    \frac{1}{n^{1-r}}\,\sup_{t\in(0,1)}\,Z_{0,n}(t)\longrightarrow 0\qquad\pr\mbox{-a.s.}
\end{equation}
In Sections \ref{Auxiliary results, Part I} we will collect some elementary properties of $Z_{p,q}(t)$. In Section \ref{Auxiliary results, Part II} we will prove some nontrivial properties of $Z_{p,q}(t)$. Finally, in Section \ref{proof of mzlln for ep -- actual proof} we will prove (\ref{convergence of interest - modi}).


\subsection{Auxiliary results, Part I}\label{Auxiliary results, Part I}

Of course, for every $p\in\N_0$ and $q,u\in\N$ with $q<u$, and every $t\in(0,1)$, the elementary inequality
\begin{equation}\label{lemma on Zpqst - eq - 4}
    Z_{p,u}(t)\,\le\,Z_{p,q}(t)\,+\,Z_{p+q,u-q}(t)
\end{equation}
holds; see also \cite[p.\,333]{Philipp1977}. Let
\begin{equation}\label{def N n}
    N_n\,:=\,\Big\lfloor\frac{\log n}{\log 2}\Big\rfloor
\end{equation}
be the largest $N\in\N_0$ with $2^N\le n$. Then $n$ can be represented as
\begin{equation}\label{represent n}
    n\,=\,2^{N_n}\,+\,\sum_{j=1}^{N_n}h_j(n)\,2^{j-1}
\end{equation}
for suitable $h_j(n)\in\{0,1\}$, $j=1,\ldots,N_n$. Equation (\ref{represent n}) and a repeated application of (\ref{lemma on Zpqst - eq - 4}) yield that for every $n\in\N$ and $t\in(0,1)$
\begin{equation}\label{lemma on Zpqst corollary - eq - 1}
    Z_{0,n}(t)
    \,\le\, Z_{0,2^{N_n}}(t)\,+\,\sum_{j=1}^{N_n}Z_{2^{N_n}+b_j(n)2^j,2^{j-1}}(t)
\end{equation}
holds for suitable integers $b_j(n)\in\{0,\ldots,2^{N_n-j}-1\}$, $j\in\{1,\ldots,N_n\}$.


\subsection{Auxiliary results, Part II}\label{Auxiliary results, Part II}

Lemma \ref{borel cantelli application} below will be crucial for the central part of the proof of Theorem \ref{mzlln for ep unif} (cf.\ Section \ref{proof of mzlln for ep -- actual proof}). For the proof of Lemma \ref{borel cantelli application} we will need the following lemma, which in turn is an immediate consequence of Proposition 7.1 in \cite{Rio2000} and Markov's inequality.

\begin{lemma}\label{probability bound}
For all $p\in\N_0$, $q\in\N$ and $x>0$,
\begin{equation}\label{probability bound - eq}
    \pr\Big[q^{-1/2}\sup_{t\in(0,1)}Z_{p,q}(t)\ge x\Big]\,\le\,\frac{1}{x^2} \Big(1+4\sum_{i=0}^{q-1}\alpha(i)\Big)(2+\log{q})^2.
\end{equation}
\end{lemma}

Now, let $R>r$ be sufficiently close to $r$ (to be concretized later on) and $\beta>0$ be sufficiently close to zero (to be concretized later on). For every $N\in\N$, define the event
$$
    F_N\,:=\,\Big\{\sup_{t\in(0,1)}Z_{0,2^N}(t)\,\ge\,2^{N(1-R)}\Big\}.
$$
For every $N\in\N$, $j\in\{1,\ldots,N\}$ and $b\in\{0,\ldots,2^{N-j}-1\}$ define the event
$$
    H_N(j,b)\,:=\,\Big\{\sup_{t\in(0,1)}Z_{2^{N}+b2^j,2^{j-1}}(t)\,\ge\,2^{N(1-R)}\,2^{-\beta(N-j)}\Big\}.
$$
Moreover, for every $N\in\N$ define the event
$$
    H_N \, := \, \bigcup_{j=1}^{N}\,\bigcup_{b=0}^{2^{N-j}-1}\,H_N(j,b).
$$

\begin{lemma}\label{borel cantelli application}
$\pr\big[\limsup_{N\to\infty} F_N\big]=\pr\big[\limsup_{N\to\infty} H_N\big]=0$. In particular, $\pr$-a.s.\ there are some constants $K_1,K_2>0$ such that
$$
    \sup_{t\in(0,1)}Z_{0,2^N}(t)\,\le\,K_1\,2^{N(1-R)}
$$
for all $N\in\N$, and
$$
    \sup_{t\in(0,1)}Z_{2^{N}+b2^j,2^{j-1}}(t)\,\le\,K_2\,2^{N(1-R)}\,2^{-\beta(N-j)}
$$
for all $N\in\N$, $j\in\{1,\ldots,N\}$ and $b\in\{0,\ldots,2^{N-j}-1\}$.
\end{lemma}

\begin{proof}
By Lemma \ref{probability bound} and the assumption $\alpha(i)\le C\,i^{-\vartheta}$,
\begin{eqnarray*}
    \pr\big[F_N\big]
    & = & \pr\Big[2^{-N/2}\sup_{t\in(0,1)}Z_{0,2^N}(t)\,\ge\,2^{N(1/2-R)}\Big]\\
    & \le &\frac{1}{2^{N(1-2R)}}\Big(1+4\sum_{i=0}^{2^N-1}C\,i^{-\vartheta}\Big)(2+\log{2^N})^2\\
    & \le & K\,2^{N(2R-\vartheta)}N^2
\end{eqnarray*}
for some finite constant $K>0$, where we assumed without loss of generality that $\vartheta\in(0,1)$. Choosing $R$ sufficiently close to $r$, and taking the assumption $\vartheta>2r$ into account, we obtain $\sum_{N=1}^\infty\pr[F_N]<\infty$. Now the Borel-Cantelli lemma yields $\pr[\limsup_{N\to\infty} F_N]=0$.

Again by Lemma \ref{probability bound} and the assumption $\alpha(i)\le C\,i^{-\vartheta}$,
\begin{eqnarray*}
    \pr\big[H_N(j,b)\big]
    & = & \pr\Big[2^{-(j-1)/2}\sup_{t\in(0,1)}Z_{2^N+b2^j,2^{j-1}}(t)\,\ge\,2^{-(j-1)/2}2^{N(1-R)}2^{-\beta(N-j)}\Big]\\
    & \le &\frac{1}{2^{-(j-1)}2^{2N(1-R)}2^{-2\beta(N-j)}}\Big(1+4\sum_{i=0}^{2^{j-1}-1}C\,i^{-\vartheta}\Big)(2+\log{2^{j-1}})^2\\
    & = & K\,2^{j(2-2\beta-\vartheta)}\,2^{-N(2-2R-2\beta)}\,j^2
\end{eqnarray*}
for some finite constant $K>0$, where we again assumed without loss of generality that $\vartheta\in(0,1)$. Therefore,
\begin{eqnarray*}
    \pr\big[H_N\big]
    & \le & K\,2^{-N(2-2R-2\beta)}\sum_{j=1}^{N}\,\sum_{b=0}^{2^{N-j}-1}2^{j(2-2\beta-\vartheta)}\,j^2\\
    & \le & K\,2^{-N(2-2R-2\beta)}\sum_{j=1}^{N}2^{N-j}\,2^{j(2-2\beta-\vartheta)}\,j^2\\
    & \le & K'\,2^{-N(1-2R-2\beta)}2^{N(1-\beta-\vartheta)}\\
    & = & K'\,2^{-N(\vartheta-2R-\beta)}
\end{eqnarray*}
for some finite constant $K'>0$. Choosing $R$ sufficiently close to $r$, choosing $\beta$ sufficiently close to zero, and taking the assumption $\vartheta>2r$ into account, we obtain $\sum_{N=1}^\infty\pr[H_N]<\infty$. Now the Borel-Cantelli lemma yields $\pr[\limsup_{N\to\infty} H_N]=0$.
\end{proof}


\subsection{Completion of the proof of Theorem \ref{mzlln for ep unif}}\label{proof of mzlln for ep -- actual proof}

We now prove (\ref{convergence of interest - modi}). By (\ref{lemma on Zpqst corollary - eq - 1}) and the definition of $N_n$ as the largest $N\in\N_0$ with $2^N\le n$ (cf.\ (\ref{def N n})), we have
\begin{eqnarray*}
    \frac{1}{n^{1-r}}\sup_{t\in(0,1)}Z_{0,n}(t)
    & \le & \frac{1}{n^{1-r}}\sup_{t\in(0,1)}Z_{0,2^{N_n}}(t)\,+\,\frac{1}{n^{1-r}}\sum_{j=1}^{N_n}\sup_{t\in(0,1)}Z_{2^{N_n}+b_j(n)2^j,2^{j-1}}(t)\\
    & =: & I_{n,1}\,+\,I_{n,2}
\end{eqnarray*}
for suitable $b_j(n)\in\{0,\ldots,2^{N_n-j}-1\}$. In the sequel we will show that $I_{n,1}$ and $I_{n,2}$ converge to zero $\pr$-a.s. This will complete the proof of Theorem \ref{mzlln for ep unif}.

As for $I_{n,1}$, we observe that by Lemma \ref{borel cantelli application} there is $\pr$-a.s.\ a constant $K_1>0$ such that $I_{n,1}\le n^{r-1}K_12^{N_n(1-R)}=K_1\,n^{r-R}$
for all $n\in\N$. Since $R>r$, the summand $I_{n,1}$ thus converges to zero $\pr$-a.s.

As for $I_{n,2}$, we observe that by Lemma \ref{borel cantelli application} there is $\pr$-a.s.\ a constant $K_2>0$ such that
\begin{eqnarray*}
    I_{n,2}
    & \le & \frac{1}{2^{N_n(1-r)}}\sum_{j=1}^{N_n}K_2\,2^{N_n(1-R)}\,2^{-\beta(N_n-j)}\\
    & \le & K_2\,2^{-N_n(R-r)}\sum_{j=0}^{N_n-1}2^{-\beta j}
\end{eqnarray*}
holds for all $n\in\N$. Since $R>r$, the summand $I_{n,2}$ thus converges to zero $\pr$-a.s. This completes the proof of Theorem \ref{generalized GC - dependent}


\section{Proof of Theorem \ref{generalized GC - dependent - ordinary}}\label{appendix a2}

Without loss of generality we assume $x_\phi=0$. So $\phi$ can be seen as a nonincreasing function on $[-\infty,0]$. We will only show that
$$
    \sup_{x\in(-\infty,0]}|\widehat F_n(x)-F(x)|\phi(x)\longrightarrow 0\qquad\pr\mbox{-a.s.}
$$
The analogous result for the positive real line can be shown in the same way. We will proceed in three steps, where we will combine arguments of \cite{van der Vaart 1998}--\cite{van der Vaart Wellner 1996} (Steps 1--2) with Rio's SLLN for $\alpha$-mixing data (Step 3). The latter can be found in \cite[Theorem 1\,(ii)]{Rio1995} and will be recalled in the following theorem. As before, the rightcontinuous inverse $h^\rightarrow$ of any nonincreasing function $h:\R_+\to[0,1]$ will be defined by $h^\rightarrow(y):=\sup\{x\in\R_+:h(x)>y\}$, $y\in[0,1]$, with the convention $\sup\emptyset:=0$.

\begin{theorem}{\em (Rio)}\label{Rios SLLN}
Let $\xi_1,\xi_2,\ldots$ be identically distributed random variables on some probability space $(\Omega,{\cal F},\pr)$ with $\ex[|\xi_1|]<\infty$. Suppose that $(\xi_i)$ is $\alpha$-mixing with mixing coefficients $(\alpha(n))$, and let $\alpha(y):=\alpha(\lfloor y\rfloor)$ be the \cadlag\ extension of $\alpha(\cdot)$ from $\N$ to $\R_+$. Let $G$ be the distribution function of $|\xi_1|$, and set $\overline G:=1-G$. If
\begin{equation}\label{Rios SLLN - cond}
    \int_0^1 \log\Big(1+\alpha^\rightarrow(y/2)\Big)\,\overline{G}\,^\rightarrow(y)\,dy\,<\,\infty,
\end{equation}
then $\frac{1}{n}\sum_{i=1}^n(\xi_i-\ex[\xi_i])\rightarrow 0$ $\pr$-a.s.
\end{theorem}

{\em Step 1.} Let $L^1(d\I)$ be the space of all Lebesgue integrable functions on $[0,1]$, and $[l,u]:=\{f\in L^1(d\I):l\le f\le u\}$ be the bracket of two functions $l,u\in L^1(d\I)$ with $l\le u$ pointwise. For any $\varepsilon>0$, a bracket $[l,u]$ is called $\varepsilon$-bracket if $\int_0^1(u-l)\,d\I<\varepsilon$. Set
$$
    w(t)\,:=\,\phi(F^\leftarrow(t))\,\eins_{[0,F(0)]}(t),\qquad t\in[0,1].
$$
Since our assumption $\int_{-\infty}^\infty\phi\,dF<\infty$ implies $\int_0^1w\,d\I<\infty$, we can find as in \cite[Example 19.12]{van der Vaart 1998} a finite partition $0={t_0^\varepsilon}<{t_1^\varepsilon}<\cdots<{t_{m_\varepsilon}^\varepsilon}=1$ of $[0,1]$ such that $[l_i^\varepsilon,u_i^\varepsilon]$ with
\begin{eqnarray*}
    l_i^\varepsilon(\cdot) & := & w(t_i^\varepsilon)\eins_{[0,t_{i-1}^\varepsilon]}(\cdot)\\
    u_i^\varepsilon(\cdot) & := & w(t_{i-1}^\varepsilon)\eins_{[0,t_{i-1}^\varepsilon]}(\cdot)\,+\,w(\cdot)\,\eins_{(t_{i-1}^\varepsilon,t_i^\varepsilon]}(\cdot)
\end{eqnarray*}
($i=1,\ldots,m_\varepsilon$) are $\varepsilon$-brackets in $L^1(d\I)$ covering the class ${\cal E}_w:=\{w_s:s\in[0,1]\}$ of functions
$$
    w_s(\cdot):=w(s)\eins_{[0,s]}(\cdot).
$$

{\em Step 2.} By the usual quantile transformation, we can find a sequence of $U[0,1]$-random variables $U_{1},U_{2},\ldots$ (possibly on an extension $(\overline\Omega,\overline{\cal F},\overline{\pr})$ of the original probability space $(\Omega,{\cal F},\pr)$) such that the sequence $(U_{i})$ has the same mixing coefficients (under $\overline{\pr}$) as the sequence $(X_i)$ under $\pr$ and such that the corresponding empirical distribution function $\widehat G_n$ satisfies $\widehat F_n=\widehat G_{n}\circ F$ $\overline{\pr}$-a.s. Here we will show as in the proof of Theorem 2.4.1 in \cite{van der Vaart Wellner 1996} that
\begin{equation}\label{appendix a2 - eq 1}
    \sup_{x\le 0}\big|\widehat F_n(x)-F(x)\big|\phi(x)
    \,\le\,\max_{i=1,\ldots,m_\varepsilon}\,\max\Big\{\int_0^1 u_i^\varepsilon\,d(\widehat G_n-\I)\,;\,\int_0^1 l_i^\varepsilon\,d(\I-\widehat G_n)\Big\}\,+\,\varepsilon
\end{equation}
for every $\varepsilon>0$. Since
\begin{eqnarray*}
    \sup_{x\le 0}\big|\widehat F_n(x)-F(x)\big|\phi(x)
    & = & \sup_{x\le 0}\big|\widehat G_n(F(x))-F(x)\big|\phi(x)\\
    & \le & \sup_{s\in(0,1)}|\widehat G_n(s)-s|\,w(s)\\
    & = & \sup_{s\in(0,1)}\Big|\int_0^1 w_s\,d\widehat G_n-\int_0^1w_s\,d\I\Big|,
\end{eqnarray*}
for (\ref{appendix a2 - eq 1}) it suffices to show that
\begin{eqnarray}
    \lefteqn{\sup_{s\in(0,1)}\Big|\int_0^1 w_s\,d\widehat G_n-\int_0^1w_s\,d\I\Big|}\nonumber\\
    & \le & \max_{i=1,\ldots,m_\varepsilon}\,\max\Big\{\int_0^1 u_i^\varepsilon\,d(\widehat G_n-\I)\,;\,\int_0^1 l_i^\varepsilon\,d(\I-\widehat G_n)\Big\}\,+\,\varepsilon.
    \label{appendix a2 - eq 1 - alt}
\end{eqnarray}
To prove (\ref{appendix a2 - eq 1 - alt}), we note that for every $s\in[0,1]$ there is some $i_s\in\{1,\ldots,m_\varepsilon\}$ such that $w_s\in[l_{i_s}^\varepsilon,u_{i_s}^\varepsilon]$; cf.\ Step 1. Therefore, since $[l_{i_s}^\varepsilon,u_{i_s}^\varepsilon]$ is an $\varepsilon$-bracket,
\begin{eqnarray*}
    \int_0^1 w_s\,d\widehat G_n-\int_0^1 w_s\,d\I
    & \le & \int_0^1 u_{i_s}^\varepsilon\,d\widehat G_n-\int_0^1 w_s\,d\I\\
    & = & \int_0^1 u_{i_s}^\varepsilon\,d(\widehat G_n-\I)+\int_0^1 (u_{i_s}^\varepsilon-w_s)\,d\I\\
    & \le & \int_0^1 u_{i_s}^\varepsilon\,d(\widehat G_n-\I)+\int_0^1 (u_{i_s}^\varepsilon-l_{i_s}^\varepsilon)\,d\I\\
    & \le & \max_{i=1,\ldots,m_\varepsilon}\int_0^1 u_i^\varepsilon\,d(\widehat G_n-\I)\,+\,\varepsilon.
\end{eqnarray*}
Analogously we obtain
\begin{eqnarray*}
    \int_0^1 w_s\,d\widehat G_n-\int_0^1 w_s\,d\I & \ge & -\Big(\max_{i=1,\ldots,m_\varepsilon}\int_0^1 l_i^\varepsilon\,d(\I-\widehat G_n)\,+\,\varepsilon\Big).
\end{eqnarray*}
That is, (\ref{appendix a2 - eq 1}) holds true.

{\em Step 3.} Because of (\ref{appendix a2 - eq 1}), for (\ref{generalized GC - dependent - ordinary - eq}) to be true it suffices to show that both $\int_0^1 l_i^\varepsilon\,d(\I-\widehat G_n)$ and $\int_0^1 u_i^\varepsilon\,d(\widehat G_n-\I)$ converge $\overline\pr$-a.s.\ to zero for every $i=1,\ldots,m_\varepsilon$. The second convergence follows from the representation
\begin{eqnarray*}
    \int_0^1 u_i^\varepsilon\,d(\widehat G_n-\I)
    & = & \frac{1}{n}\sum_{j=1}^n\Big(w(t_{i-1}^\varepsilon)\eins_{[0,t_{i-1}^\varepsilon]}(U_j)-\ex_{\overline\pr}\Big[w(t_{i-1}^\varepsilon)\eins_{[0,t_{i-1}^\varepsilon]}(U_1)\Big]\Big)\\
    &  & +\,\frac{1}{n}\sum_{j=1}^n\Big(w(U_j)\eins_{(t_{i-1}^\varepsilon,t_i^\varepsilon]}(U_j)-\ex_{\overline\pr}\Big[w(U_1)\,\eins_{(t_{i-1}^\varepsilon,t_i^\varepsilon]}(U_1)\Big]\Big)
\end{eqnarray*}
and Theorem \ref{Rios SLLN}, noting that (\ref{generalized GC - dependent - ordinary - cond}) implies (\ref{Rios SLLN - cond}) for both $\xi_j:=w(t_{i-1}^\varepsilon)\eins_{[0,t_{i-1}^\varepsilon]}(U_j)$ and $\xi_j:=w(U_j)\eins_{(t_{i-1}^\varepsilon,t_i^\varepsilon]}(U_j)$. The verification of the first convergence is even easier. This completes the proof of Theorem \ref{generalized GC - dependent - ordinary}.



\end{document}